\documentclass[10pt, a4paper]{amsart}
 
 \usepackage[top=2.5cm,bottom=2.5cm,left=3.7cm,right=3.7cm]{geometry} 
\usepackage[pdftex]{graphicx}
\usepackage{amsmath}
\usepackage{amsfonts}
\usepackage{algorithmic}
\usepackage{array}

\newcommand{\set}[1]{\left\{ #1 \right\}}

\newcommand{\isom}{\xrightarrow{\sim}}
\newcommand{\Id}{\mathrm{Id}}
\newcommand{\N}{\mathbb{N}}
\newcommand{\Z}{\mathbb{Z}}

\newcommand{\R}{\mathbb{R}}

\newcommand{\F}{\mathbb{F}}
\newcommand{\A}{\mathbb{A}}
\newcommand{\PP}{\mathbb{P}}
\newcommand{\GL}{\mathrm{GL}}

\newcommand{\vu}{\mathbf{u}}
\newcommand{\vh}{\mathbf{h}}

\newcommand{\PGL}{\mathrm{PGL}}
\newcommand{\PGr}{\mathbb{G}\mathrm{r}}
\newcommand{\ceil}[1]{\left\lceil #1 \right\rceil}
\newcommand{\floor}[1]{\left\lfloor #1 \right\rfloor}

\theoremstyle{plain}
\newtheorem{theorem}{Theorem}
\newtheorem{proposition}[theorem]{Proposition}
\newtheorem{corollary}[theorem]{Corollary}
\newtheorem{lemma}[theorem]{Lemma}

\theoremstyle{definition}

\newtheorem{definition}[theorem]{Definition}

\newtheorem{example}[theorem]{Example}

\newtheorem{remark}[theorem]{Remark}

\numberwithin{equation}{section}
\numberwithin{theorem}{section}

\begin{document}

\title[Fractional Jumps]{Full Orbit Sequences in Affine Spaces via Fractional Jumps and Pseudorandom Number Generation}

%%%%%%%%%%%%%%%%%%%%%%%%%%%%%%%%%%%%%%%%%%%%%%%%%%%
%AUTHORS
%%%%%%%%%%%%%%%%%%%%%%%%%%%%%%%%%%%%%%%%%%%%%%%%%%%
\author{Federico Amadio Guidi}
\address{Mathematical Institute, University of Oxford, Oxford, UK}
\email{federico.amadio@maths.ox.ac.uk}
\author{Sofia Lindqvist}
\address{Mathematical Institute, University of Oxford, Oxford, UK}
\email{sofia.lindqvist@maths.ox.ac.uk}
\author[Giacomo Micheli]{Giacomo Micheli$^*$}
\thanks{$^*$ Corresponding author.}
\address{Mathematical Institute, University of Oxford, Oxford, UK}
\email{giacomo.micheli@maths.ox.ac.uk}

%%%%%%%%%%%%%%%%%%%%%%%%%%%%%%%%%%%%%%%%%%%%%%%%%%%
%%%%%%%%%%%%%%%%%%%%%%%%%%%%%%%%%%%%%%%%%%%%%%%%%%%

\maketitle

\begin{abstract}
Let $n$ be a positive integer. In this paper we provide a general theory to produce full orbit sequences in the affine $n$-dimensional space over a finite field. For $n=1$ our construction covers the case of the Inversive Congruential Generators (ICG).
In addition, for $n>1$ we show that the sequences produced using our construction are easier to compute than ICG sequences. Furthermore, we prove that they have the same discrepancy bounds as the ones constructed using the ICG.
\end{abstract}
{\footnotesize\noindent\textbf{MSC2010 subject classification}: 11B37, 15B33, 11T06,  11K38, 11K45, 11T23, 65C10, 37P25.}

\section{Introduction}

In recent years there has been a great interest in the construction of discrete dynamical systems with given properties (see for example \cite{bib:eich93,bib:EHHW98,bib:HBM17,bib:ost10,bib:OPS10,bib:OS10degree,bib:OS10length}) both for applications (see for example \cite{bib:BW05,bib:chou95,bib:eich91,bib:eich92,bib:GPOS14, bib:NS02, bib:NS03, bib:TW06,bib:winterhof10}) and for the purely mathematical interest that these objects have (see for example \cite{bib:eich91,bib:EMG09,bib:ferraguti2016existence,bib:FMS16,bib:FMS17,bib:GSW03}).
This paper deals with the problem of finding discrete dynamical systems which can be new candidates for pseudorandom number generation.

Let us denote the set of natural numbers by $\N$. Given a finite set $S$, a sequence $\{a_m\}_{m\in \N}$ of elements in $S$ is said to have \emph{full orbit} if for any $s\in S$ there exists $m\in \N$ such that $a_m=s$.

Let $q$ be a prime power, $\F_q$ be the finite field of cardinality $q$, and $n$ be a positive integer. In this paper  we produce maps $\psi:\F_q^n \rightarrow \F_q^n$ such that 
\begin{itemize}
\item the sequences $\{\psi^m(0)\}_{m\in \N}$ have full orbit (whenever this property is verified, we say that the map $\psi$ is \emph{transitive}), 

\item the sequences constructed from $\psi$ have nice discrepancy bounds, analogous to those constructed from an Inversive Congruential Generator (ICG),
\item they are very inexpensive to iterate: if $n>1$ they are asymptotically less expensive than an  ICG for the same bitrate.
\end{itemize}
In addition, such maps  can be described using quotients of degree one polynomials.

From a purely theoretical point of view related to the full orbit property, one of the reasons why such constructions are interesting is that one cannot build transitive affine maps (i.e. of the form $x\mapsto Ax+b$, with $A$ an invertible $n\times n$ matrix and $b$ an $n$-dimensional vector) unless either $n=1$ and $q$ is prime, or $n=2$ and $q=2$ (see Theorem \ref{affine_transitivity_theorem}). 

For $n=1$ our construction covers the well-studied case of the ICG, for which we obtain easy proofs of classical facts (see for example Remark \ref{remarkICGfullorbit}). In fact, we fit the theory of full orbit sequences in a much wider context, where tools from projective geometry can be used to establish properties of the sequences produced with our method (see for example Proposition \ref{theorem_uniformity}).

Let us now summarise the results of the paper. The main tool we use to construct full orbit sequences is the notion of fractional jump of projective maps, which is described in Section \ref{affine_jumps}. 
With such a notion we are able to produce maps in the affine space which can be guaranteed to be transitive when they are fractional jumps of transitive projective maps. In Section \ref{transitivity_projective} we characterise transitive projective maps using the notion of projective primitivity for polynomials (see Definition \ref{projectively_primitive_polynomial}). In Section \ref{uniformity} we show that whenever our sequences come from  the iterations of transitive projective automorphisms, they behave quite uniformly with respect to proper projective subspaces (i.e. not many consecutive element in the sequence can lie in a proper subspace of the projective space). This fact (and in particular Proposition \ref{theorem_uniformity}) will allow us in Section \ref{explicit} to give an explicit description of the fractional jump of a transitive projective map, finally leading to the new explicit constructions of full orbit sequences promised earlier. In turn, such a description and the theory developed in Section \ref{transitivity_projective} allow us to prove the discrepancy bounds of Theorem \ref{thm:discrepancy} in Section \ref{discrepancy}. In Section \ref{computation} we show the computational advantage of our approach compared to the classical ICG one.
Finally, we include some conclusions which summarise the results of the paper.

%%%%%%%%%%%%%%%%%%%%%%%%%%%%%%%%%%%%%%%%%%%%%%%%%%%
%%%%%%%%%%%%%%%%%%%%%%%%%%%%%%%%%%%%%%%%%%%%%%%%%%%

\subsection*{Notation}

Let us denote the set of natural numbers by $\N$, and the ring of integers by $\Z$. For a commutative ring with unity $R$, let us denote by $R^*$ the group of invertible elements of $R$.

We denote by $\F_q$ the finite field of cardinality $q$, which will be fixed throughout the paper, and by $\overline{\F}_q$ an algebraic closure of $\F_q$. Given an integer $n \geq 1$, we often denote the $n$-dimensional affine space $\F_q^n$ by $\A^n$. The $n$-dimensional projective space over the finite field $\F_q$ is denoted by $\PP^n$. Also, we denote by $\PGr (d, n)$ the set of $d$-dimensional projective subspaces of $\PP^n$.

We denote by $\F_q[x_1,...,x_n]$ the ring of polynomials in $n$ variables with coefficients in $\F_q$. For a polynomial $a \in \F_q[x_1,...,x_n]$ we denote by $\deg a$ its total degree, which we will simply call its degree. Also, for $b\in\F_q[x_1,\dots,x_n]$ we let $V(b)$ denote the set of points $x\in \A^n$ such that $b(x)=0$.

We denote by $\GL_n (\F_q)$ the general linear group over the field $\F_q$, i.e.  the group of $n \times n$ invertible matrices with entries in $\F_q$, and by $\PGL_{n+1} (\F_q)$ the group of automorphisms of $\PP^{n}$. Recall that $\PGL_{n+1} (\F_q)$ can be identified with the quotient group $\GL_{n+1} (\F_q) / \F_q^*\Id$, where $\F_q^*\Id$ is just the subgroup of nonzero scalar multiples of the identity matrix $\Id$. Given a matrix $M \in \GL_{n+1} (\F_q)$, we denote by $[M]$ its class in $\PGL_{n+1} (\F_q)$.

Let $X$ be either $\A^n$ or $\PP^n$. We will say that a map $f : X \rightarrow X$ is \emph{transitive}, or equivalently that it \emph{acts transitively on $X$}, if for any $x, y \in X$ there exists an integer $i \geq 0$ such that $y = f^i (x)$. Equivalently, $f$ is transitive if and only if for any $x \in X$ the sequence $\{ f^m(x)\}_{m \in \N}$ has full orbit, that is $\{ f^m (x) \, : \, m \in \N \} = X$. A map $f:\A^n\rightarrow \A^n$ is said to be affine if there exist $A\in \GL_n(\F_q)$ and $b\in \F_q^n$ such that $f(x)=Ax+b$ for any $x\in \A^n$. 

Let $G$ be a group acting on a set $S$. The orbit of an element $s\in S$ will be denoted by $\mathcal O(s)$. For any element $g\in G$, let us denote by $o(g)$ the order of $g$ in $G$.

We write $f\ll g$ or $f=O(g)$ to mean that for some positive constant $C$ it holds that $|f|\le Cg$. The notation $f\ll_\delta g$ or $f=O_\delta(g)$ means the same, but now the constant $C$ may depend on the parameter $\delta$.
For any real vector $\vh=(h_1,\dots, h_n)$, we write 
$\|\vh\|_{\infty}=
\max\{|h_j|\, : \, j\in \{1,\dots, n\}\}$.
Finally, for any prime $p$ and any $z \in \Z$ we write $e_p (z) = \exp (2 \pi i z / p)$.

%%%%%%%%%%%%%%%%%%%%%%%%%%%%%%%%%%%%%%%%%%%%%%%%%%%
%%%%%%%%%%%%%%%%%%%%%%%%%%%%%%%%%%%%%%%%%%%%%%%%%%%
\section{Fractional jumps} \label{affine_jumps}

Fix the standard projective coordinates $X_0, \ldots, X_n$ on $\PP^n$, and the canonical decomposition 
\begin{equation} \label{decomposition}
\PP^n = U \cup H,
\end{equation}
where
\begin{align*}
U &= \set{[X_0: \ldots: X_n] \in \PP^n \, : \, X_n \neq 0}, \\
H &= \set{[X_0: \ldots: X_n] \in \PP^n \, : \, X_n = 0}.
\end{align*}

There is a natural isomorphism of the affine $n$-dimensional space into $U$ given by
\begin{equation} \label{pi_definition}
\pi : \A^n \isom U, \quad (x_1, \ldots, x_n) \mapsto [x_1: \ldots: x_n: 1],
\end{equation}
%with inverse 
%\begin{equation*}
%\pi^{-1} : U \isom \A^n, \quad [X_0: \ldots: X_n] \mapsto \rnd{\frac{X_0}{X_n}, \ldots, \frac{X_{n-1}}{X_n}}.
%\end{equation*}

Let now $\Psi$ be an automorphism of $\PP^n$. We give the following definitions:

\begin{definition}
For $P \in U$, the \emph{fractional jump index of $\Psi$ at $P$} is
\begin{equation*}
\mathfrak{J}_P = \min \set{k \geq 1 \, : \, \Psi^k (P) \in U}.
\end{equation*}
\end{definition}

\begin{remark}
The fractional jump index $\mathfrak{J}_P$ is always finite, as it is bounded by the order of $\Psi$ in $\PGL_{n+1} (\F_q)$.
\end{remark}

\begin{definition}
The \emph{fractional jump of $\Psi$} is the map
\begin{equation*}
\psi : \A^n \rightarrow \A^n, \quad x \mapsto \pi^{-1} \Psi^{\mathfrak{J}_{\pi(x)}} \pi (x).
\end{equation*}
\end{definition}

Roughly speaking, the purpose of defining this new map is to avoid the points which are mapped outside $U$ via $\Psi$. This is done simply by iterating $\Psi$ until $\Psi(\pi (x))$ ends up again in $U$. In this definition, $\pi$ is simply used to obtain the final map defined over $\A^n$ instead of $U$. A priori, one of the issues here is that a global description of the map might be difficult to compute, as in principle it depends on each of the $x\in \A^n$. It is interesting to see that this does not happen in the case in which $\Psi$ is transitive on $\PP^n$:
in fact, we will show in Section \ref{explicit} that there always exists a set of indices $I$, a disjoint covering $\set{U_i}_{i \in I}$ of $\A^n$, and a family $\set{f^{(i)}}_{i \in I}$ of rational maps of degree $1$ on $\A^n$ such that
\begin{enumerate}
\item[i)] $|I| \leq n+1$,
\item[ii)] $f^{(i)}$ is well-defined on $U_i$ for every $i \in I$,
\item[iii)] $\psi (x) = f^{(i)} (x)$ if $x \in U_i$.
\end{enumerate}
That is, $\psi$ can be written as a multivariate linear fractional transformation on each $U_i$. In addition, for any fixed $i\in \{1,\dots n+1\}$, all the denominator of the $f^{(i)}$'s will be equal.

 \begin{example} \label{inversive}
Let $n=1$. For $a\in \F_q^*$ and $b\in \F_q$ and
 \begin{equation*}
 \Psi ([X_0: X_1]) = [b X_0 + a X_1: X_0]
 \end{equation*}
 we get the case of the inversive congruential generator. In fact, the fractional jump index of $\Psi$ is given by
 \begin{equation*}
\mathfrak{J}_P = \begin{cases}
1, & \text{if } P \neq [0, 1], \\
2, & \text{if } P = [0, 1],
\end{cases}
\end{equation*}
and $\Psi^2 ([0, 1]) = [b, 1]$. Therefore, the fractional jump $\psi$ of $\Psi$ is defined on the covering $\set{U_1, U_2}$, where $U_1 = \A^1 \setminus \set{0}$ and $U_2 = \set{0}$, by
\begin{equation*}
\psi (x) = \begin{cases}
\frac{a}{x} + b, & \text{if } x \neq 0, \\
b, & \text{if } x = 0.
\end{cases}
\end{equation*}
The inversive sequence is then given by $\set{\psi^m (0)}_{m \in \N}$, which has full orbit under suitable assumptions on $a$ and $b$ (see for example \cite[Lemma FN]{bib:chou95}).
\end{example}

\begin{remark} \label{remark_transitivity_affine_jump}
 Let $\Psi$ be an automorphism of $\PP^n$ and let $\psi$ be its fractional jump. It is immediate to see that if $\Psi$ acts transitively on $\PP^n$ then $\psi$ acts transitively on $\A^n$.
 \end{remark}
 
For the case of $n = 1$, the next proposition shows that the notion of transitivity for $\Psi$ and its fractional jump $\psi$ are actually equivalent, under the additional assumption that $\Psi$ sends a point of $U$ to a point of $H$ (which is equivalent to ask that the induced map on $\A^1$ is not affine).

 \begin{proposition} \label{transitivity_affine_jump_P1}
Let $\Psi$ be an automorphism of $\PP^1$ and let $\psi$ be its fractional jump. Assume that $\Psi$ sends a point of $U$ to the point at infinity. Then, $\Psi$ acts transitively on $\PP^1$ if and only if $\psi$ acts transitively on $\A^1$.
\end{proposition}

\begin{proof}
As already stated in Remark \ref{remark_transitivity_affine_jump}, if $\Psi$ is transitive on $\PP^1$ then $\psi$ is obviously transitive on $\A^1$. Conversely, assume that $\psi$ is transitive on $\A^1$. Consider the decomposition $\PP^1 = U \cup H$ of $\PP^1$ as in \eqref{decomposition}. Since $n = 1$, we have $H = \set{P_0}$, for $P_0 = [1 : 0]$. Since there exists $P_1 \in U$ such that $\Psi (P_1) = P_0$, we have that $\Psi^2 (P_1) = \Psi (P_0) \in U$, as otherwise the point $P_0$ would have two preimages under $\Psi$, which is not possible as $\Psi$ is an automorphism, and so in particular a bijection. We have to prove that given $P, Q \in \PP^1$ there exists an integer $i\geq 0$ such that $Q = \Psi^i (P)$. Assume that $P$ and $Q$ are distinct, as otherwise we can simply set $i = 0$. We distinguish two cases: either $P, Q \in U$, or one of the two, say $P$, is equal to $P_0$ and $Q \in U$. In the first case, the claim follows by transitivity of $\psi$. In the second case, reduce to the previous case by considering $\Psi (P_0), Q \in U$.
\end{proof}

One can  actually prove that affine transformations of $\A^n$ are never transitive, unless restrictive conditions on $q$ and $n$ apply. Actually, the result that follows will not be used in the rest of the paper but provides additional motivation for the study of fractional jumps of projective maps and for completeness we include its proof.

\begin{theorem} \label{affine_transitivity_theorem}
There is no affine transitive transformation of $\A^n$ unless $n = 1$ and $q$ is prime, or $q = 2$ and $n = 2$, with explicit examples in both cases.
\end{theorem}

\begin{proof}
For convenience of notation, in this proof we will identify the points of $\A^n$ with columns vectors in $\F_q^n$.
Let us first deal with the pathological cases.
For $n=1$ it is trivial to observe that $x\mapsto x+1$ has full orbit if and only if $q$ is prime.
For $n=2$ and $q=2$, we get by direct check that the map
\begin{equation*}
\varphi \begin{pmatrix} x_1 \\ x_2 \end{pmatrix} = \begin{pmatrix} 1 & 1 \\ 0 & 1 \end{pmatrix} \cdot \begin{pmatrix} x_1 \\ x_2 \end{pmatrix} + \begin{pmatrix} 1 \\ 1 \end{pmatrix}, \quad \begin{pmatrix} x_1 \\ x_2 \end{pmatrix} \in \F_2^2,
\end{equation*}
has full orbit.

Let $\varphi$ be an affine transformation of the $n$-dimensional affine space over $\F_q$. Then, by definition there exist $A \in \GL_n (\F_q)$ and $b \in \F_q^n$ such that
\begin{equation*}
\varphi (x) = A x + b, \quad x \in \F_q^n.
\end{equation*}

Assume by contradiction that $\varphi$ is transitive, so that the order $o (\varphi)$ of $\varphi$ is $q^n$. Denote by $p$ the characteristic of $\F_q$. We firstly prove that the order $o (A)$ of $A$ in $\GL_n (\F_q)$ is $q^n / p$. Then, we will show how this will lead to a contradiction.

Let $j$ be the smallest integer such that
\begin{equation*}
\varphi^j (x) = x + c, \quad \text{for all }x \in \F_q^n,
\end{equation*}
for some $c \in \F_q^n$. As
\begin{equation} \label{explicit_affine}
\varphi^j (x) = A^j x + \sum_{i = 0}^{j-1} A^i b, \quad x \in \F_q^n,
\end{equation}
we get $o(A) = j$. If $c = 0$, then $o(\varphi) = j = o (A) \leq q^n - 1$, so that $\varphi$ cannot be transitive. We then have $c \neq 0$. By \eqref{explicit_affine}, we get $\varphi^{j p} = \Id$, therefore $o(\varphi)\mid jp$. We now prove that $o(\varphi) = jp$. Write $o(\varphi) = j s + r$, with $r < j$. Then, we have
\begin{align*}
\varphi^{j s + r} (x) &= \varphi^r (x) + s c \\
&= A^r x + v, \quad x \in \F_q^n,
\end{align*}
for a suitable $v \in \F_q^n$. Since $\varphi^{j s + r} = \Id$, we get that $A^r x + v = x$ for all $x \in \F_q^n$, and so we must have $r = 0$ and $v=0$. It follows that $\varphi^{j s} (x) = x + s c = x$ for all $x \in \F_q^n$, which gives $p \mid s$, as $c \neq 0$, so that we get $p \leq s$, and then $o(\varphi) = j s \geq jp$. Therefore we conclude that $o (\varphi) = jp$. As $\varphi$ is assumed to be transitive, we have that $j p = q^n$, and so $o (A) = j = q^n / p$.
Essentially, what we have proved up to now is that, if such a transitive affine map $\varphi(x)=Ax+b$ exists, then it must have the property that $o (A) = q^n / p$.

Let $\mu_A (T) \in \F_q [T]$ be the minimal polynomial of $A$. By the fact that $o (A) = q^n / p$ we get
\begin{equation*}
\mu_A (T) \mid T^{q^n / p} - 1 = (T-1)^{q^n / p}.
\end{equation*}
Then, $\mu_A (T) = (T - 1)^d$, for some $d \leq n$, as the degree of the minimal polynomial is less than or equal to the degree of the characteristic polynomial by Cayley-Hamilton. From basic ring theory, one gets that the order of $A$ in $\GL_n (\F_q)$ is equal to the order of the class $\overline{T}$ of $T$ in the quotient ring $(\F_q[T] / (\mu_A (T)))^* = (\F_q[T] / ((T - 1)^d))^*$. 
Let us now assume $q^n / p^2 \geq n$. In this case we have
\begin{equation*}
\overline{T}^{q^n / p^2} = (\overline{T}-1)^{q^n / p^2} + 1 = 1,
\end{equation*}
as $q^n / p^2 \geq n \geq  d$. Therefore,  $o (A)=o(\overline T)\leq q^n/p^2 < q^n / p$ from which the contradiction follows.

Therefore we can restrict to the case  $q^n / p^2 < n$. It is easy to see that this inequality forces $q=p$: in fact if $q=p^k$ and $k\geq 2$, then  $q^n/p^2=p^{kn-2}\geq p^{2n-2}\geq 4^{n-1}\geq  n$. Therefore, the only uncovered cases are in correspondence with the solutions of $p^{n-2}<n$, which consist only of the following: $n=3$ and $p=2$, or $n=1$ and $p$ any prime, or $n=2$ and $p$ any prime.
For $n=3$ and $p=2$ an exhaustive computation shows that there is no transitive affine map.
Also, we already know that in the case $n=1$ and $p$ any prime we have such a transitive map, as this is one of the pathological cases.
For the case $n=2$ we argue as follows. Let
\[\varphi(x)=Ax+b\]
be such a transitive affine map.
Clearly $A\in \GL_2(\F_p)$ must be different from the identity matrix, as otherwise $\varphi$ cannot have full orbit. So the minimal polynomial of $A$ is different from $T-1$. On the other hand, the minimal polynomial of $A$ must divide $(T-1)^d$. Since $n=2$, we have that $d=2$. In $\GL_2(\F_p)$ having minimal polynomial $(T-1)^2$ forces a matrix to be conjugate to a single Jordan block of size $2$ with eigenvalue $1$, hence there exists $C\in \GL_2(\F_p)$ such that
\[CAC^{-1}= \begin{pmatrix} 1 & 1 \\ 0 & 1 \end{pmatrix} \]
Let us now consider again the map $\varphi$. Clearly, $\varphi$ is transitive if and only
 if the map $\widetilde \varphi=C\varphi C^{-1}$ is. For any $x\in \F_p^2$ we have that $\widetilde \varphi(x)=C(AC^{-1}x+b)$. Therefore
 the map $\widetilde \varphi$ can be written as
\[\widetilde \varphi\begin{pmatrix} x_1 \\ x_2 \end{pmatrix}=\begin{pmatrix} 1 & 1 \\ 0 & 1 \end{pmatrix} \cdot \begin{pmatrix} x_1 \\ x_2 \end{pmatrix} + \begin{pmatrix} r \\ s  \end{pmatrix}.\]
for some $r,s\in \F_p$.
We will now prove that $\widetilde \varphi^p\begin{pmatrix} r \\ s\end{pmatrix}=\begin{pmatrix} r \\ s \end{pmatrix}$ so that $\varphi$ cannot be transitive, as starting form $c:=\begin{pmatrix} r \\ s\end{pmatrix}$ only visits $p$ points.
\begin{align*}
\widetilde \varphi^p\begin{pmatrix} r \\ s\end{pmatrix}&=\sum^{p}_{i=0} \begin{pmatrix} 1 & 1 \\ 0 & 1 \end{pmatrix}^ic \\
&= c+\sum^p_{i=1}\begin{pmatrix} 1 & i \\ 0 & 1 \end{pmatrix} \cdot \begin{pmatrix} r \\ s \end{pmatrix} \\
&= c+\sum^p_{i=1}\begin{pmatrix} r+is \\ s \end{pmatrix}=c+ \begin{pmatrix} s\sum^p_{i=1}i \\ 0 \end{pmatrix}
\end{align*}
But now the sum $\sum^p_{i=1}i$ is different from zero if and only if $p= 2$. Therefore, such a transitive map could exist only for $p=n=2$. Since we already provided such an example of transitive map, the proof of the theorem is now concluded.

\end{proof}

%%%%%%%%%%%%%%%%%%%%%%%%%%%%%%%%%%%%%%%%%%%%%%%%%%%
%%%%%%%%%%%%%%%%%%%%%%%%%%%%%%%%%%%%%%%%%%%%%%%%%%%

\section{Transitive actions via projective primitivity} \label{transitivity_projective}

In this section we characterise transitive projective automorphisms.
\begin{definition} \label{projectively_primitive_polynomial}
A polynomial $\chi(T) \in \F_q [T]$ of degree $m$ is said to be \emph{projectively primitive} if the two following conditions are satisfied:
\begin{enumerate}
\item[i)] $\chi(T)$ is irreducible over $\F_q$,
\item[ii)] for any root $\alpha$ of $\chi(T)$ in $\F_{q^m} \cong \F_q[T] / (\chi(T))$ the class $\overline{\alpha}$ of $\alpha$ in the quotient group $G = \F_{q^m}^* / \F_q^*$ generates $G$.
\end{enumerate} 
\end{definition}

\begin{remark}
Note that if a polynomial $\chi(T)\in \F_q[T]$ of degree $m$ is primitive, i.e. it is irreducible and any of its roots in $\F_{q^m} \cong \F_q[T] / (\chi(T))$ generates the multiplicative group $\F_{q^m}^*$, then it is obviously projectively primitive. The class of projectively primitive polynomials is in general larger than then the class of primitive polynomials: for example take the polynomial $\chi(T)=T^3+T+1\in \F_5[T]$. One can check that this polynomial is irreducible but is not primitive. In fact, let $\alpha$ be a root of $\chi(T)$. Since  $\alpha^{62}=1\in \F_5[T] / (\chi(T))\cong \F_{5^3}$, and therefore $o(\alpha)\neq 5^3-1=124$, we have that $\chi(T)$ is not primitive. On the other hand $G=\F_{5^3}^* / \F_5^*$ has prime cardinality  equal to $|G|=(5^3-1)/(5-1)=31$ and $\overline{\alpha}\neq 1\in G$. It follows immediately that $\overline \alpha$ has to be a generator of $G$.  
\end{remark}

\begin{remark}
Let $M, M' \in \GL_{n+1} (\F_q)$ be such that $[M] = [M']$ in $\PGL_{n+1} (\F_q)$, and let $\chi_M (T), \chi_{M'} (T) \in \F_q [T]$ be their characteristic polynomials. It is immediate to see that $\chi_M (T)$ is projectively primitive if and only if $\chi_{M'} (T)$ is projectively primitive.
\end{remark}

We are now ready to give a full characterisation of transitive projective automorphisms on $\PP^n$.

\begin{theorem} \label{transitivity_characterisation}
Let $\Psi$ be an automorphism of $\PP^n$. Write $\Psi = [M] \in \PGL_{n+1} (\F_q)$ for some $M \in \GL_{n+1} (\F_q)$. Then, $\Psi$ acts transitively on $\PP^n$ if and only if the characteristic polynomial $\chi_M (T) \in \F_q [T]$ of $M$ is projectively primitive.
\end{theorem}

\begin{proof}
For simplicity of notation, set $N = |\PP^n| = (q^{n+1}-1) / (q-1)$. Assume that $\chi_M (T)$ is projectively primitive. Now we prove that for any $P \in \PP^n$ we have that $\Psi^k (P) \neq P$ for $k \in \set{1, \ldots, N -1}$. Suppose by contradiction that there exists $P_0 \in \PP^n$ such that $\Psi^k (P_0) = P_0$ for some $k \in \set{1, \ldots, N-1}$. Let $v_0 \in \F_q^{n+1} \setminus \set{0}$ be a representative of $P_0$. Then, there exists $\lambda \in \F_q^*$ such that
\begin{equation*}
M^k v_0 = \lambda v_0.
\end{equation*}
This means that $v_0$ is an eigenvector for the eigenvalue $\lambda$ of $M^k$, which implies that $\lambda = \alpha^k$ for some root $\alpha$ of $\chi_M (T)$ in $\F_{q^{n+1}}$. But now, the class $\overline{\alpha}^k$ of $\alpha^k$ in $G = \F_{q^{n+1}}^* / \F_q^*$ is $\overline{\lambda} = \overline{1}$, contradicting the hypothesis that $\overline{\alpha}$ generates $G$.

Conversely, assume that $\Psi$ is transitive, so that for any $P \in \PP^n$ we have that $\Psi^k (P) \neq P$ for $k \in \set{1, \ldots, N -1}$. Let $\alpha$ be a root of $\chi_M(x)$ in its splitting field  and $h$ be a positive integer such that $\F_{q^h}\cong\F_q(\alpha)$. Clearly $1\leq h\leq n+1$. We also have that $\alpha \neq 0$ as $\det M \neq 0$, and $\alpha \notin \F_q^*$ as otherwise $M v_0 = \alpha v_0$ for some eigenvector $v_0 \in \F_q^{n+1} \setminus \set{0}$  for the eigenvalue $\alpha$, so that $\Psi (P_0) = P_0$ for $P_0$ the class of $v_0$ in $\PP^n$, in contradiction with the fact that $\Psi$ is transitive. Let $d$ be the order of the class $\overline{\alpha}$ of $\alpha$ in $\F_{q^h}^* / \F_q^*$. Then, there exists $\lambda \in \F_q^*$ such that $\alpha^d = \lambda$. Now, $\alpha^d$ is an eigenvalue of $M^d$, and so $M^d v_1 = \lambda v_1$ for some eigenvector $v_1\in \F_q^{n+1} \setminus \set{0}$ for the eigenvalue $\alpha^d$. Thus, $\Psi^d (P_1) = P_1$ for $P_1$ the class of $v_1$ in $\PP^n$, and so $d=N$ by the transitivity of $\Psi$. Therefore we have $(q^{n+1}-1)/(q-1)=N=d\leq (q^{h}-1)/(q-1)$. Now, since $h\leq n+1$, this forces $h = n+1$, so that $\chi_M$ is irreducible, which together with $d=N$ gives projective primitivity for $\chi_M$, as we wanted. 
\end{proof}

\begin{remark}\label{remarkICGfullorbit}
When $n = 1$, our approach gives immediately the criterion to get maximal period for inversive congruential generators, see for example \cite[Lemma FN]{bib:chou95}. To see this, set $\Psi$ and $\psi$ as in Example \ref{inversive}, so that $\set{\psi^m (0)}_{m \in \N}$ is an inversive sequence. By Proposition \ref{transitivity_affine_jump_P1}, transitivity of $\Psi$ and $\psi$ are equivalent. If $\chi (T) = T^2 - a T - b$ is irreducible, then $\psi$ acts transitively on $\A^1$ if and only if the class $\overline{\alpha}$ of a root $\alpha$ of $\chi(T)$ in $G = \F_{q^{2}}^* / \F_q^*$ generates $G$, which is itself equivalent to the fact that $\alpha^{q-1}$ has order $q+1$ in $\F_{q^2}^*$ (which is in fact the condition given in \cite[Lemma FN]{bib:chou95}).
\end{remark}

%%%%%%%%%%%%%%%%%%%%%%%%%%%%%%%%%%%%%%%%%%%%%%%%%%%
%%%%%%%%%%%%%%%%%%%%%%%%%%%%%%%%%%%%%%%%%%%%%%%%%%%

\section{Subspace Uniformity} \label{uniformity}

In this section we show that sequences associated to iterations of transitive projective maps behave ``uniformly'' with respect to subspaces, i.e. not too many consecutive points can lie in the same projective subspace of $\PP^n$. 
%Notice that for any positive integer $d$ and any $W\in \PGr (d, n)$ we have that $\Psi(W)\in \PGr (d, n)$.
%We first need to prove the following ancillary result.
%
%\begin{lemma} \label{faithful}
%Let $\Psi$ be a transitive automorphism of $\PP^n$. For any $d \in \set{1, \ldots, n-1}$, $\Psi$ has no fixed points on $\PGr (d, n)$, i.e. $\Psi(W)\neq W$ for any $W\in \PGr (d, n)$.
%\end{lemma}
%
%\begin{proof}
%Let $d \in \set{1, \ldots, n-1}$ and assume there exists $W \in \PGr (d, n)$ such that $\Psi (W) = W$. Then, we get a well defined map
%\begin{equation*}
%\Psi \! \mid_W : W \rightarrow W.
%\end{equation*}
%
%Let $P \in W$. The orbit $\mathcal{O}_W (P)$ of $P$ with respect to the action of $\Psi \! \mid_W$ on $W$ is a subset of $W$, and so
%\begin{align*}
%|\mathcal{O}_W (P)| &\leq |W| \\
%&= \frac{q^{d+1} -1}{q-1}.
%\end{align*}
%
%But now $\mathcal{O}_W (P)$ is the same as the orbit $\mathcal{O} (P)$ of $P$ with respect to the action of $\Psi$ on $\PP^n$, which is the full $\PP^n$ by transitivity. Therefore
%\begin{align*}
%|\mathcal{O}_W (P)| &= |\mathcal{O} (P)| = |\PP^n| \\
%&= \frac{q^{n+1} -1}{q-1} > \frac{q^{d+1} -1}{q-1},
%\end{align*}
%a contradiction.
%\end{proof}

%We are now ready to prove that sequences of many consecutive points given by iterations of a transitive $\Psi$ cannot lie in a proper subspace. 
More precisely, we have the following:

\begin{proposition} \label{theorem_uniformity}
Let $\Psi$ be a transitive automorphism of $\PP^n$. For any $P \in \PP^n$ and any $d \in \set{1, \ldots, n-1}$ there is no $W \in \PGr (d, n)$ such that $\Psi^i (P) \in W$ for all $i \in \set{0, \ldots, d+1}$.
\end{proposition}

\begin{proof}

Suppose by contradiction that there exists a projective subspace $W$ of dimension $d$ such that there exists $P\in \PP^n$ such that $\Psi^i (P) \in W$ for all $i\in \set{0, \ldots, d+1}$. Let $W'$ be the subspace of $\F_q^{n+1}$ whose projectification is $W$, and let $v \in \F_q^{n+1} \setminus \set{0}$ be a representative for $P$. Let also $M\in \GL_{n+1}(\F_q)$ be a representative for $\Psi$. Consider now the smallest integer $h$ such that $M^h v$ is linearly dependent on $\{M^i v \, : \, i\in \{0,\dots h-1\}\}$ over $\F_q$. Since $M^iv$ is contained in $W'$ for any $i\in \{0,\dots d+1\}$, and $W'$ has dimension $d+1$, then we have that $h$ is at most  $d+1$. Therefore, $M^hv$ can be rewritten in terms of lower powers of $M$, which in turn forces the span of $\{ M^iv \, : \, i\in \{0,\dots, h-1\}\}$ over $\F_q$ to be an invariant space for $M$. It follows that the characteristic polynomial of $M$ has a non-trivial factor of degree less than or equal to $d$. Since $d\leq n$, we have the claim, as by Theorem \ref{transitivity_characterisation} the characteristic polynomial of $M$ has to be irreducible.
\end{proof}

\begin{remark}
This result is  optimal with respect to $d$, as for any set $S$ of $d+2$ points of $\PP^n$ there always exists a $W\in \PGr (d+1, n)$ containing $S$.
\end{remark}

Fix the canonical decomposition $\PP^n = U \cup H$ as in \eqref{decomposition}.

\begin{corollary} \label{bound_affine_jump_index}
Let $\Psi$ be a transitive automorphism of $\PP^n$. For any $P \in U$ the fractional jump index $\mathfrak{J}_P$ of $\Psi$ at $P$ is bounded by $n+1$.
\end{corollary}

\begin{proof}
Assume by contradiction $\mathfrak{J}_P \geq n+2$. Then, setting $P' = \Psi (P)$ we get $\Psi^i (P') \in H$ for all $i \in \set{0, \ldots, n}$. But we have $H \in \PGr (n-1, n)$, and so this violates Proposition \ref{theorem_uniformity}. 
\end{proof}

%%%%%%%%%%%%%%%%%%%%%%%%%%%%%%%%%%%%%%%%%%%%%%%%%%%
%%%%%%%%%%%%%%%%%%%%%%%%%%%%%%%%%%%%%%%%%%%%%%%%%%%

\section{Explicit description of fractional jumps} \label{explicit} 

Let $\Psi$ be an automorphism of $\PP^n$. In this section we will give an explicit description of the fractional jump $\psi$ of $\Psi$.

First of all, fix homogeneous coordinates $X_0, \ldots, X_n$ on $\PP^n$, fix the canonical decomposition $\PP^n = U \cup H$ as in \eqref{decomposition} and the map $\pi$ as in \eqref{pi_definition}, and write $\Psi\in \PGL_{n+1} (\F_q)$ as
\begin{equation*}
\Psi = [F_0: \ldots: F_n],
\end{equation*}
where each $F_j$ is an homogeneous polynomial of degree $1$ in $\F_q [X_0, \ldots, X_n]$.
Fix now affine coordinates $x_1, \ldots, x_n$ on $\A^n$, and for each $j \in \set{1, \ldots, n}$ set
\begin{equation} \label{rational_functions}
f_j (x_1, \ldots, x_n) = \frac{F_{j-1} (x_1, \ldots, x_n, 1)}{F_n (x_1, \ldots, x_n, 1)}.
\end{equation}
Let $K = \F_q (x_1, \ldots, x_n)$ be the field of rational functions on $\A^n$. Then, \eqref{rational_functions} defines elements $f_j \in K$ for $j \in \set{1, \ldots, n}$, and $f_\Psi = (f_1, \ldots, f_n) \in K^n$.
In turn this process defines a map
\begin{equation} \label{map_pgl}
\imath : \PGL_{n+1} (\F_q) \rightarrow K^n, \quad \Psi \mapsto f_\Psi.
\end{equation}
It is easy to see that this map is well defined and for any element $f=(f_1,\dots,f_n)$ in the image  of $\imath$  all the denominators of the $f_j$'s are equal. It also holds that $\imath(\Psi\circ \Phi)=\imath(\Psi)\circ \imath(\Phi)$, where the composition in $K^n$ is defined in the obvious way, i.e. just plugging in the components of $\imath (\Phi)$ in the variables of $\imath(\Psi)$.

Let us go back to $f_\Psi = (f_1, \ldots, f_n) \in K^n$ for a fixed automorphism $\Psi$. For any $i \geq 1$, let us define $f^{(i)} = \imath(\Psi^{i})$. For each $f^{(i)}$ and for each $j\in \{1,\dots,n\}$, write the $j$-th component of $f^{(i)}$ as
\begin{equation*}
f^{(i)}_j = \frac{a^{(i)}_j}{b^{(i)}_j}, \quad \text{for } a^{(i)}_j, b^{(i)}_j \in \F_q [x_1, \ldots, x_n].
\end{equation*}

As we already observed, for fixed $i\geq 1$ all the $b^{(i)}_j$'s are equal, so we can set  $b^{(i)}=b^{(i)}_1$. Define now
\begin{align*}
V_0 &= \A^n, \\
V_i &= \bigcap_{k = 1}^i V(b^{(k)}), \quad \text{for } i \geq 1.
\end{align*}

These sets will be the main ingredient in the definition of the covering mentioned in Section \ref{affine_jumps}. The following result characterises the $V_i$'s in terms of the position of a bunch of iterates of $\Psi$.

\begin{lemma} \label{characterisation_vanishing_loci}
Let $x \in \A^n$, and $P = \pi (x) \in U$. Then, $x \in V_i$ if and only if $\Psi^k (P) \in H$ for $k \in \set{1, \ldots, i}$.
\end{lemma}

\begin{proof}
By definition, $x \in V_i$ if and only if $x \in V(b^{(k)})$ for every $k \in \set{1, \ldots, i}$, which means $b^{(k)} (x) = 0$ for every $k \in \set{1, \ldots, i}$. Now, $b^{(k)} (x) = 0$ if and only if the last component of $\Psi^k (P)$ is zero, which is equivalent to the condition $\Psi^k (P) \in H$.
\end{proof}

\begin{definition}
Define the \emph{absolute fractional jump index $\mathfrak{J}$ of $\Psi$} to be the quantity
\begin{equation*}
\mathfrak{J} = \max \set{\mathfrak{J}_P \, : \, P \in U}.
\end{equation*}
\end{definition}

When $\Psi$ is transitive, Corollary \ref{bound_affine_jump_index} ensures that $\mathfrak{J} \leq n+1$. We will now show that the absolute jump index equals the number of non empty $V_i$'s. 

\begin{proposition}
We have that
\begin{equation*} \label{absolute_jump_index}
\min \set{i \in \N \, : \, V_i = \emptyset} = \mathfrak{J}.
\end{equation*}
\end{proposition}

\begin{proof}
Set $i_0 = \min \set{i \in \N \, : \, V_i = \emptyset}$. In order to show that $i_0 \leq \mathfrak{J}$, it is enough to prove that $V_{\mathfrak{J}} = \emptyset$. Assume that there exists $x \in V_{\mathfrak{J}}$. Then, if $P = \pi (x)$, we have by Lemma \ref{characterisation_vanishing_loci} that $\Psi^j (P) \in H$ for $j \in \set{1, \ldots, \mathfrak{J}}$, and so the jump index $\mathfrak{J}_P$ must be strictly greater than  $\mathfrak{J} $, a contradiction.

Conversely, in order to show that $\mathfrak J \leq i_0$, it is enough to prove that $V_{\mathfrak J -1}\neq \emptyset $. To do so, take $P_0 \in U$ for which $\mathfrak{J}_{P_0} = \mathfrak{J}$. Then $\Psi^k(P_0)\in H$ for any $k\in \{1,\dots \mathfrak{J}-1\}$. Let $x_0 = \pi^{-1} (P_0)$. Then, by Lemma \ref{characterisation_vanishing_loci} we have $x_0 \in V_{\mathfrak{J}-1}$.
\end{proof}

Define now
\begin{equation*}
U_i = V_{i-1} \setminus V_i, \quad \text{for } i \in \set{1, \ldots, \mathfrak{J}}.
\end{equation*}

Thus, for $I = \set{1, \ldots, \mathfrak{J}}$, the family $\set{U_i}_{i \in I}$ is a disjoint covering of $\A^n$ and each $f^{(i)}$ is a rational map of degree $1$ on $\A^n$. Also, we observe that by construction $f^{(i)}$ is well defined on $U_i$, so that the fractional jump is defined as 
\begin{equation*}
\psi (x) = f^{(i)} (x), \quad \text{if } x \in U_i.
\end{equation*}

To clarify this contruction, we now give an explicit description of a fractional jump over $\A^2$.
\begin{example}
Let $q = 101$ and $n = 2$. Consider the automorphism of $\PP^2$ defined by
\begin{align*}
\Psi([X_0: X_1: X_2]) &= [F_0 : F_1 : F_2] \\
&= [X_0 + 2 X_2: 3 X_1+4 X_2: 4X_0 + 2 X_1 + 3 X_2].
\end{align*}

Notice that
\begin{equation*}
M = \begin{pmatrix} 1 & 0 & 2 \\ 0 & 3 & 4 \\ 4 & 2 & 3 \end{pmatrix}
\end{equation*}
is a representative of $\Psi$ in $\GL_3 (\F_{101})$. The characteristic polynomial $\chi_M (T) \in \F_{101} [T]$ of $M$ is given by
\begin{equation*}
\chi_M (T) = T^3 - 7 T^2 - T + 23,
\end{equation*}
which is irreducible over $\F_{101}$. Now, as
\begin{align*}
\frac{q^{n+1}-1}{q-1} &= \frac{101^3-1}{101-1} \\
&= 10303
\end{align*}
is prime, any irreducible polynomial of degree $3$ in $\F_{101}[T]$ is projectively primitive. By Theorem \ref{transitivity_characterisation} we have that $\Psi$ is transitive on $\PP^n$. Since $n=2$ and $\Psi$ is transitive, by Proposition \ref{absolute_jump_index} and the definition of the $U_i$'s  we know that the 
fractional jump of $\Psi$ will be defined using at most $U_1,U_2,U_3$.

As in \eqref{rational_functions}, we consider rational functions
\begin{align*}
f_1(x_1, x_2) &= \frac{x_1+2}{4x_1+2 x_2 +3}, \\
f_2(x_1, x_2) &= \frac{3x_2+4}{4x_1+2 x_2 +3}.
\end{align*}
in $\F_{101} (x_1, x_2)$, and set $f = (f_1, f_2) \in \F_{101}(x_1, x_2)^2$. Given the definition of $f$, we have
\begin{align*}
V_1 &= V(4x_1+2 x_2 +3), \\
U_1 &= \A^2 \setminus V_1.
\end{align*}

Let now $f^{(1)} = f$ and $f^{(2)} = f \circ f = (f_1^{(2)}, f_2^{(2)}) \in \F_{101}(x_1, x_2)^2$, where
\begin{align*}
f_1^{(2)}(x_1, x_2) &= f_1 (f_1 (x_1, x_2), f_2 (x_1, x_2)) \\
&= \frac{9x_1 + 4x_2 + 8}{16 x_1 + 12 x_2 +25}, \\
f_2^{(2)}(x_1, x_2) &= f_2 (f_1 (x_1, x_2), f_2 (x_1, x_2)) \\
&= \frac{16 x_1 + 17 x_2 + 24}{16 x_1 + 12 x_2 +25}.
\end{align*}
Define
\begin{align*}
V_2 &= V_1 \cap V(16 x_1 + 12 x_2 +25) = \set{(64, 22)}, \\
U_2 &= V_1 \setminus V_2.
\end{align*}

Finally, let $f^{(3)} = f \circ f \circ f = (f_1^{(3)}, f_2^{(3)}) \in \F_{101}(x_1, x_2)^2$, where
\begin{align*}
f_1^{(3)}(x_1, x_2) &= f_1 (f_1^{(2)} (x_1, x_2), f_2^{(2)} (x_1, x_2)) \\
&= \frac{41 x_1 + 28 x_2 - 43}{15 x_1 - 15 x_2 - 47}, \\
f_2^{(3)}(x_1, x_2) &= f_2 (f_1^{(2)} (x_1, x_2), f_2^{(2)} (x_1, x_2)) \\
&= \frac{11 x_1 - 2 x_2 - 30}{15 x_1 - 15 x_2 - 47},
\end{align*}
and $U_3 =V_2= \set{(64, 22)}$, since $V_3=V_2 \cap V(15x_1-15x_2-47)=\emptyset$.

By construction, $\A^2 = U_1 \cup U_2 \cup U_3$, and therefore we are ready to describe the fractional jump $\psi$ of $\Psi$ as

\begin{equation*}
\psi (x_1, x_2) = \begin{cases}
f^{(1)} (x_1, x_2), & \text{if }(x_1, x_2) \in U_1, \\
f^{(2)} (x_1, x_2), & \text{if }(x_1, x_2) \in U_2, \\
f^{(3)} (x_1, x_2), & \text{if }(x_1, x_2) \in U_3.
\end{cases}
\end{equation*}
Notice that $f^{(3)} (64, 22) = (63, 78)$, and so $\psi (x_1, x_2) = (63, 78)$ if $(x_1, x_2) \in U_3 = \set{(64, 22)}$.

\end{example}
 
%%%%%%%%%%%%%%%%%%%%%%%%%%%%%%%%%%%%%%%%%%%%%%%%%%%
%%%%%%%%%%%%%%%%%%%%%%%%%%%%%%%%%%%%%%%%%%%%%%%%%%%

\section{The discrepancy of fractional jump sequences}\label{discrepancy}

In the context of pseudorandom number generation, it is of interest to say something about the distribution of a sequence. A statistic that is of particular interest is the discrepancy of a sequence, of which we recall the  definition below. The goal of this section is to show that for sequences generated by fractional jumps one can prove the same discrepancy bounds as for the sequences generated by the ICG.
For simplicity, we let $q=p$ be prime. We assume the set $ \F_p\cong \Z/p\Z$ to be represented by $\{0,1,\dots, p-1\}\subseteq \Z$ as in \cite{shparlinski10}. For $x\in \F_p$ we then write $\frac{x}{p}$ for the corresponding element in $\frac{1}{p}\Z\subseteq \R$.

For a sequence
\begin{equation*}
\Gamma = \{(\gamma_{m,0},\dots,\gamma_{m,s-1})\}_{m=0}^{N-1}
\end{equation*}
of $N$ points in $[0,1)^s$, for $s\in \N$, the \emph{discrepancy} of $\Gamma$ is defined by
\begin{equation*}
D_\Gamma = \sup_{B\subseteq [0,1)^s} \bigg|\frac{T_\Gamma(B)}{N}-|B|\bigg|,
\end{equation*}
where the supremum is taken over boxes $B$ of the form
\begin{equation*}
B = [\alpha_1,\beta_1)\times \dots \times [\alpha_s,\beta_s) \subseteq [0,1)^s,
\end{equation*}
and $T_\Gamma(B)$ denotes the number of points of $\Gamma$ which lie inside $B$. 

For a sequence $\{u_m\}_{m \in \N}$ of points in $\F_p$ the main interest lies in bounding the discrepancy of the sequence
\[ \Big(\frac{u_{m}}{p},\frac{u_{m+1}}{p},\dots,\frac{u_{m+s-1}}{p}\Big)_{m=0}^{N-1}\]
for $s\ge 1$. In the case of a sequence generated by an ICG such a bound was given in \cite{shparlinski10}. The goal of this section is to extend the results in \cite{shparlinski10} to give discrepancy bounds for full orbit sequences generated by fractional jumps also in the case where the dimension $n$ satisfies $n>1$.

Given a fractional jump $\psi:\F_p^n\to \F_p^n$ and an initial value $x\in \F_p^n$ we define the sequence $\{ \vu_m (x) \}_{m \in \N}$ of points in $\F_p^n$ by setting $\vu_0(x) = x$ and 
\begin{equation*}
\vu_m(x) = \psi^{m}(x), \quad \text{for } m \geq 1.
\end{equation*}
We also define the \emph{snake sequence} $\{ v_m (x) \}_{m \ge 1}$ of points in $\F_p$ by setting
\[ (v_{kn+1}(x),v_{kn+2}(x),\dots,v_{(k+1)n}(x)) = \vu_{k}(x), \quad \text{for } k \in \N.\]
Let $D_{s,\psi}(N;x)$ denote the discrepancy of the sequence
\[ \Big(\frac{v_{m+1}}{p},\frac{v_{m+2}}{p},\dots,\frac{v_{m+s}}{p}\Big)_{m=0}^{N-1}\]
and let $D^n_{s,\psi}(N;x)$ denote the discrepancy of the sequence
\[ \Big(\frac{\vu_m}{p},\frac{\vu_{m+1}}{p},\dots,\frac{\vu_{m+s-1}}{p}\Big)_{m=0}^{N-1}.\]
Note that in the first case the individual points of the sequence lie in $\F_p^s$, while in the second case the points of the sequence lie in $\F_p^{ns}$.

Our main result for the discrepancy $D_{s,\psi}(N;x)$ is a direct generalization of \cite[Theorem 4]{shparlinski10}, which deals with the discrepancy of a sequence generated by an ICG. We also provide the analogous bounds for the $n$-dimensional discrepancy $D_{s,\psi}^n(N;x)$.
\begin{theorem}\label{thm:discrepancy}
Let $\Psi$ be a transitive automorphism of $\PP^n$ and let $\psi$ be its fractional jump. Then for any integer $s\ge 1$ and any real $\Delta >0$, for all but $O(\Delta p^n)$ initial values $x\in \F_p^n$ it holds that
\[ D_{s,\psi}(N;x) \ll_{s,n} (\Delta^{-2/3}N^{-1/3}+ p^{-1/4}\Delta^{-1})(\log N)^s \log p\]
and
\[ D_{s,\psi}^n(N;x) \ll_{s,n} (\Delta^{-2/3}N^{-1/3}+ p^{-1/4}\Delta^{-1})(\log N)^{sn} \log p\]
for all $N$ with $1\le N\le p^n$. 
\end{theorem}

The proof of Theorem \ref{thm:discrepancy} follows the same lines as the proof of \cite[Theorem 4]{shparlinski10}, but with Lemma \ref{lem:2nd-moment} below extending \cite[Lemma 1]{shparlinski10} to $n>1$.

In the proofs we will make use of the Koksma--Sz\"{u}sz inequality as well as the Bombieri--Weil bound.

\begin{theorem}[{\cite[Theorem 1.21]{drmota}}]\label{thm:koksma}
For any integer $H\ge 1$, the discrepancy $D_\Gamma$ of the sequence $\Gamma = (\gamma_{m,0},\dots,\gamma_{m,s-1})_{m=0}^{N-1}$ satisfies
\[ D_\Gamma \ll \frac{1}{H} + \frac{1}{N}\sum_{0<\| \vh \|_\infty \le H} \frac{1}{\rho(\vh)} \bigg| \sum_{m=0}^{N-1} \exp\Big( 2\pi i \sum_{j=0}^{s-1} h_j \gamma_{m,j} \Big)\bigg|, \]
where $\rho(\vh) = \prod_{j=0}^{s-1} \max\{|h_j|,1\}$ for $\vh=(h_0,\dots, h_{s-1})\in \Z^s$.
\end{theorem}

\begin{theorem}[{\cite[Theorem 2]{moreno}}]\label{thm:weil}
Let $f/g$ be a rational function over $\F_p$ with $\deg(f)>\deg(g)$. Suppose that $f/g$ is not of the form $h^p-h$, where $h$ is a rational function over $\overline{\F}_p$.
Then
\[ \bigg|\sum_{\substack{x\in \F_p:\\g(x)\ne 0}} e_p\left( \frac{f(x)}{g(x)} \right) \bigg| \le (\deg(f) + v-1)p^{1/2},\]
where $v$ is the number of distinct roots of $g$ in $\overline{\F}_p$.
\end{theorem}

We will also need to use the explicit description of $\psi$ given in Section \ref{explicit} to describe powers of $\psi$, which is done in the next lemma.

\begin{lemma}\label{lem:explicit}
Let $\Psi$ be a transitive automorphism of $\PP^n$ and let $\psi$ be its
fractional jump. Then there are polynomials $a^{(i)}_j, b^{(i)} \in \F_p [x_1, \ldots, x_n]$ of degree less or equal than $1$, for $i\in\{1,\dots,p-1\}$ and $j\in\{1,\dots,n\}$, with $b^{(i)}$ not
identically a constant, and such that
\[ \psi^i_j(x) = \frac{a^{(i)}_j(x)}{b^{(i)}(x)}, \quad \text{for } x\not\in \bigcup_{k=1}^i V(b^{(k)}),\]
where $\psi^i_j (x)$ denotes the $j$-th component of $\psi^i (x)$.
\end{lemma}

\begin{proof}
The functions $a_j^{(i)}, b^{(i)}$ are defined as in Section \ref{explicit}. Indeed, recall that there is a set $U_1$ and there is a rational map
\[ f^{(1)}=\Big(\frac{a^{(1)}_1}{b^{(1)}},\dots,\frac{a^{(1)}_n}{b^{(1)}}\Big) \]
of degree $1$ such that
\[ \psi(x) = f^{(1)}(x),\quad \text{for } x\in U_1 = \F_p^n\setminus V(b^{(1)}).\]
For $i\in \{1,\dots,p-1\}$ and $j\in\{1,\dots,n\}$, define the maps $a_j^{(i)},b^{(i)}$ by iterating the map $f^{(1)}$, that is
\[ f^{(i)} = (f^{(1)})^i = \left(\frac{a^{(i)}_1}{b^{(i)}},\dots,\frac{a^{(i)}_n}{b^{(i)}}\right),\quad i\in\{1,\dots,p-1\}.\]
Let us notice that in Section \ref{explicit} the function $f^{(i)}$ was used to describe the map $\psi$ on the set $U_i$. In this section we are instead using $f^{(i)}$ to describe the $i$'th iterate of $\psi$ on the set $U_1\cap \psi^{-1}(U_1) \cap \cdots \cap \psi^{-i+1}(U_1)$. In particular, Section \ref{explicit} made use of $f^{(i)}$ for $i\in\{1,\dots,n\}$, but here we instead use $f^{(i)}$ on the range $i\in \{1,\dots,p-1\}$.

To see that $b^{(i)}$ isn't identically a constant for $i \in \set{1, \ldots, p-1}$ we need to show that $f^{(i)} = \imath(\Psi^i)$, where $\imath$ is the map in \eqref{map_pgl}, is not affine. Let $M\in \GL_{n+1}(\F_p)$ be such that $\Psi = [M] \in \PGL_{n+1}(\F_p)$. 
Suppose by contradiction that $\imath(\Psi^i)$ is affine for some $i \in \set{1, \ldots, p-1}$. Since $M$ has irreducible characteristic polynomial by Theorem \ref{transitivity_characterisation}, we have that  $\F_p[M]$ is a field and in turn that $\F_p[M^i]$ is a proper subfield, therefore the minimal polynomial of $M^i$ is irreducible. Now, since $\imath(\Psi^i)$ is assumed to be affine, we have that  $\Psi^i(H)=H$, which in turn forces $M^i$ to fix a proper subspace of $\F_p^{n+1}$. This directly implies that the (irreducible) minimal polynomial of $M^i$ cannot be equal to the characteristic polynomial of $M^i$, and therefore it must have degree $d<n+1$.

We know, again by Theorem \ref{transitivity_characterisation}, that $[M]$ is a generator for the quotient group $\F_p[M]^*/ \F_p^*$ as $\Psi$ is transitive, but $[1]=[M^i]^{(p^d-1)/(p-1)}= [M]^{i (p^d-1)/(p-1)}$, so $(p^{n+1}-1)/(p-1)\mid i (p^d-1)/(p-1)$ which forces $i\geq 
(p^{n+1}-1)/(p^d-1)\geq p$, a contradiction.
\end{proof}

We are now ready to prove the technical heart of the argument.

\begin{lemma}\label{lem:technical}
Let $\Psi$ be a transitive automorphism of $\PP^n$ and let $\psi$ be its fractional jump. Then for any integers $j_0,
s\ge 1$, $d\le (p-1)n-s$ and $\vh \in \F_p^s\setminus \{0\}$ it holds that
\[ \Big|\sum_{x \in \F_p^n} e_p\Big( \sum_{j=0}^{s-1}h_j(v_{j_0+d+j}(x)-v_{j_0+j}(x))\Big)\Big| \le 3\Big(\frac{s+d}{n}+1\Big)p^{n-1}+4\Big(\frac{s}{n}+1\Big)p^{n-1/2}.
 \]
\end{lemma}

\begin{proof}
Observe first that the result is trivial for $s \ge p^{1/2}n$, so assume that $s\le p^{1/2}n$.

Let $r = \min\{j: h_j\ne 0\}$, $s'=s-r$ and $h_j'=h_{j+r}$ for $j\in\{0,\dots,s'-1\}$. Let $m= \floor{(j_0+r)/n}$. Since $\psi$ is a bijection, we can make the substitution $x'=\psi^m (x)$ and sum over $x'\in\F_p^n$ in place of summing over $x\in \F_p^n$. Then, we get $v_{i}(x')=v_{mn+i}(x) $, and so
\begin{align*}
\sum_{x\in\F_p^n} e_p\Big(\sum_{j=0}^{s-1} h_j(v_{j_0+d+j}(x)-v_{j_0+j}(x))\Big) &= \sum_{x'\in\F_p^n} e_p\Big(\sum_{j=0}^{s'-1} h_{j}'(v_{j_1+d+j}(x')-v_{j_1+j}(x'))\Big)\\
&=\sum_{x\in\F_p^n} e_p\Big(\sum_{j=0}^{s'-1} h_{j}'(v_{j_1+d+j}(x)-v_{j_1+j}(x))\Big),
\end{align*}
for some $j_1$ with $1\le j_1 \le n$, where in the last equality we have simply relabeled the summation index $x'$ to $x$, which we do for simplicity of notation. Notice that in this way we have that $h_0' \ne 0$, which was the entire point of shifting the sum.

As $d\le (p-1)n-s$ we have $j_1+j \le j_1+d+j \le
pn-1 < pn$. This means that $v_{j_1+d+j} = \psi_k^i(x)$ for some $i< p$ and some
$k\in\{1,\dots, n\}$.  An analogous statement also holds for $v_{j_1+j}$, i.e. $v_{j_1+j} = \psi_{k'}^{i'}(x)$ for some $i'< p$ and some $k'\in\{1,\dots, n\}$.
 We can therefore apply Lemma \ref{lem:explicit} to write
\[v_{in+j}(x) = \psi_j^i(x) = \frac{a_j^{(i)}(x)}{b^{(i)}(x)},\quad x\not\in \bigcup_{k=1}^i V(b^{(k)}),\]
for $i\in\{1,\dots,\floor{\frac{n+d+s-1}{n}}\}$, $j\in \{1,\dots,n\}$, and for $i=0$ we clearly have
\[ v_{j}(x) = x_j,\]
for $j \in \{1,\dots,n\}$.

Since we want to estimate the sum 
\[\Big|\sum_{x \in \F_p^n} e_p\Big( \sum_{j=0}^{s'-1}h_j'(v_{j_0+d+j}(x)-v_{j_0+j}(x))\Big)\Big|,\]
we consider for fixed $\tilde{x} = (x_1,\dots,x_{j_1-1},x_{j_1+1},\dots,x_n)\in \F_q^{n-1}$ the inner sum
\[G(x_{j_1};\tilde{x})= \sum_{j=0}^{s'-1} h_j'(v_{j_1+d+j}(x)-v_{j_1+j}(x))\]
as a function of the variable $x_{j_1}$. Since we want to apply Theorem \ref{thm:weil} to $G$ for fixed $\tilde x$ (and considered as a univariate polynomial in $x_{j_1}$) we first need to give a nice description of $G$ outside a certain set. We can do that outside of the set
\[E=\bigcup_{i=1}^{\floor{(n+d+s-1)/n}}V(b^{(i)}).\]
In fact, one may write
\[G(x_{j_1};\tilde{x})= \frac{a(x_{j_1};\tilde{x})}{b(x_{j_1};\tilde{x})} = \frac{\tilde{a}(x_{j_1};\tilde{x})}{b(x_{j_1};\tilde{x})} - h_0' x_{j_1} + c(\tilde{x}),\]
where $a,\tilde{a},b$ are polynomials and $c(\tilde{x})$ is constant with respect to $x_{j_1}$.

In order to apply Theorem \ref{thm:weil} to the sum over $x_{j_1}$, we need to check that the conditions of the theorem are verified apart from a small set $F$ of $\tilde x$'s, whose size we can estimate. To begin with we check that $\deg(a) = \deg(b)+1$. This follows immediately if either $\deg(\tilde{a})\le \deg(b)$, or if $\deg(\tilde{a})=\deg(b)+1$ and the leading coefficient in $\tilde{a}/b$ doesn't cancel the term $-h_0'x_{j_1}$. 

By considering the possible powers of $\psi$ that can appear in the definition of $G$, we see that 
\[b(x_{j_1};\tilde{x}) = \prod_{i\in I} b^{(i)}(x),\]
with the product taken over a set $I$ of $i$ satisfying
\begin{equation}
i \in \left[ \frac{j_1}{n}-1,\frac{j_1+s'-1}{n} \right)\cup\left[ \frac{j_1+d}{n}-1,\frac{j_1+d+s'-1}{n} \right) \subseteq [0, p)
\label{eq:ugly}
\end{equation}
and such that the coefficient of $x_{j_1}$ is nonzero in $b^{(i)}$.
 
Since $\tilde{a}/b$ was defined by a linear sum of rational functions of degree $1$, it follows that $\deg(\tilde{a}) \le \deg(b) + 1$. If $\deg(\tilde{a}) = \deg(b)+1$, the coefficient of the highest order term in $a$ is of the form a constant times $\prod_{i\in J} b^{(i)}(x)$ for some set $J$ of $i$ satisfying \eqref{eq:ugly} and such that $b^{(i)}$ doesn't depend on $x_{j_1}$. Think of this coefficient as a polynomial in $\tilde{x}$. In particular, there are no more than $2 \Big(\frac{s'}{n}+1 \Big)$ values of $\tilde{x}$ satisfying \eqref{eq:ugly}, and therefore the coefficient is equal to $h_0'$ for at most $2 \Big( \frac{s'}{n}+1 \Big) p^{n-2}$ values of $\tilde{x}$. We can therefore define a set $F \subseteq \F_p^{n-1}$ with
\[ |F| \le 2\Big(\frac{s'}{n}+1\Big)p^{n-2},\]
and such that $\deg(a) = \deg(b) + 1$~for $\tilde{x} \not \in F$.

Finally, for $\tilde{x}\not\in F$ we want to check that $G$ is not of the form $h^p-h$ for some rational function $h$ over $\overline{\F}_p$. Assume therefore that in fact $a/b = h^p-h$ for some rational function $h = h_1/h_2$, where $h_1$ and $h_2$ are coprime. Then ${h_2^pa = h_1^p b - h_1bh_2^{p-1}}$, and so in particular $h_2^p | b$. Note that
\[ \deg(b) \le 2(s'/n+1) < p \]
since we initially assumed that $s\le p^{1/2}n$, and so $h_2$ must be constant. This gives $a = b(c_1h_1^{p} - c_2h_1)$ for some constants $c_1,c_2$. But then $\deg(a) -\deg(b)$ is a multiple of $p$, contradicting $\deg(a) = \deg(b)+1$.

Combining all of this we may apply Theorem \ref{thm:weil} to the sum over $x_{j_1}$ to conclude that whenever $\tilde{x}\not\in F$ it holds that
\[\bigg|\sum_{x_{j_1}} e_p\Big(\frac{a(x_{j_1};\tilde{x})}{b(x_{j_1};\tilde{x})}\Big)\bigg| \le 4\Big(\frac{s}{n}+1\Big)p^{1/2},\]
where the sum is taken over values $x_{j_1}$ where $b\ne 0$. For $\tilde{x}\in F$ we have the trivial bound
\[\bigg|\sum_{x_{j_1}} e_p\Big(\frac{a(x_{j_1};\tilde{x})}{b(x_{j_1};\tilde{x})}\Big)\bigg| \le p.\]

Finally, these bounds together with the union bound 
\[ |E| \le \sum_{i=0}^{\floor{(n+d+s-1)/n}} |V(b^{(i)})| \le \left(\frac{s+d}{n}+1\right)p^{n-1}\]
and the triangle inequality give
\begin{align*}
\Big|\sum_{x\in\F_p^n} e_p(G(x_{j_1};\tilde{x}))\Big| &\le |E| + \Big|\sum_{x\not\in E} e_p\Big(\frac{a(x_{j_1};\tilde{x})}{b(x_{j_1};\tilde{x})}\Big)\Big| \\
    &\le |E| + p|F| + \Big|\sum_{\tilde{x}\not\in F}\sum_{\substack{x_{j_1}:\\ x\not\in E}} e_p\Big(\frac{a(x_{j_1};\tilde{x})}{b(x_{j_1};\tilde{x})}\Big)\Big|\\
    &\le \Big(\frac{s+d}{n}+1\Big)p^{n-1} + 2p\Big(\frac{s}{n}+1\Big)p^{n-2} + 4\Big(\frac{s}{n}+1\Big)p^{1/2}p^{n-1}\\
    &\le 3\Big(\frac{s+d}{n}+1\Big)p^{n-1}+4\Big(\frac{s}{n}+1\Big)p^{n-1/2}.
\end{align*}
\end{proof}

We now need an additional ancillary result, which will be used in the proof of the main theorem.

\begin{lemma}\label{lem:2nd-moment}
Let $\Psi$ be a transitive automorphism of $\PP^n$ and let $\psi$ be its fractional jump. Then for any integers $j_0, s \geq 1$ and $K$ with $1\le K \le p^n$, and any $\vh\in \F_p^s\setminus \{0\}$ one has
\[ \sum_{x\in \F_p^n} \bigg| \sum_{k=0}^{K-1} e_p\Big(\sum_{j=0}^{s-1}h_jv_{j_0+j+k}(x)\Big)\bigg|^2 \ll_{s,n} Kp^n + K^2p^{n-1/2}. \]
\end{lemma}
\begin{proof}
We divide into the two cases $K\le p^{1/2}$ and $K> p^{1/2}$. In the first case we have
\begin{align*}
\sum_{x\in \F_p^n} \bigg| \sum_{k=0}^{K-1} e_p\Big(\sum_{j=0}^{s-1}h_jv_{j_0+j+k}(x)\Big)\bigg|^2 = \sum_{x\in\F_p^n} \sum_{m,l=0}^{K-1}e_p\Big(\sum_{j=0}^{s-1}h_j (v_{j_0+j+m}(x)-v_{j_0+j+l}(x))\Big)\\
    \le Kp^n + 2\sum_{d=1}^{K-1}\sum_{m=0}^{K-1-d}\bigg|\sum_{x\in\F_p^n} e_p\Big(\sum_{j=0}^{s-1} h_j(v_{j_0+m+j+d}(x)-v_{j_0+m+j}(x))\Big)\bigg| ,
\end{align*}
where we have split into the cases $m=l$ and $m\ne l$. Applying Lemma \ref{lem:technical} to the innermost sum when $d\le (p-1)n-s$, and applying the trivial bound for the $O_{s,n}(1)$ remaining values of $d$ then gives that this is
\begin{align*}
&\ll_{s,n} K p^n + \sum_{d=1}^{K-1} (K-d)\left(\frac{d}{n}p^{n-1} +(s/n+1)p^{n-1/2}\right) + p^n\\
&\ll_{s,n}  Kp^n + K^3p^{n-1} + K^2p^{n-1/2}.
\end{align*}
As the middle term never dominates for the considered range of $K$, we are done in this case.

In the second case, split the sum over $k$ into at most $K/M+1$ intervals of length $M=p^{1/2}$. On each interval we bound the sum as in the first case, and so by Cauchy--Schwarz it follows that 
\begin{align*}
\sum_{x\in \F_p^n} \bigg|\sum_{k=0}^{K-1} e_p\Big(\sum_{j=0}^{s-1}h_jv_{j_0+j+k}(x)\Big)\bigg|^2 &\ll_{s,n} \left(\frac{K^2}{M^2}+1\right)(Mp^n+M^3p^{n-1}+M^2p^{n-1/2})\\
&\ll_{s,n} K^2p^{n-1/2}.
\end{align*}
\end{proof}

We are now ready to prove the main theorem.

\begin{proof}[Proof of Theorem \ref{thm:discrepancy}]
Apply Theorem \ref{thm:koksma} with $H = \floor{N/2}$ to get
\begin{equation}
D_{s,\psi}(N;x) \ll \frac{1}{N}+\frac{1}{N}\sum_{0<\| \vh \|_\infty \le N/2} \frac{1}{\rho(\vh)}\bigg| \sum_{m=0}^{N-1} e_p\Big(\sum_{j=0}^{s-1}h_j v_{m+j}(x)\Big)\bigg|.
\label{eq:bound1}
\end{equation}
Let $k\geq 1$ be an integer. Observe that if $k > N-1$, we have that
\[\bigg|\sum_{m=0}^{N-1}e_p\Big(\sum_{j=0}^{s-1}h_j v_{m+j}(x)\Big)-\sum_{m=0}^{N-1}e_p\Big(\sum_{j=0}^{s-1} h_j v_{m+j+k}(x)\Big)\bigg| \leq 2N \leq 2k,\]
if $k \leq N-1$, since the two sums in $m$ overlap in all but $2k$ terms, we have that
\begin{align*}
&\bigg|\sum_{m=0}^{N-1}e_p\Big(\sum_{j=0}^{s-1}h_j v_{m+j}(x)\Big)-\sum_{m=0}^{N-1}e_p\Big(\sum_{j=0}^{s-1} h_j v_{m+j+k}(x)\Big)\bigg| \\
& = \bigg|\sum_{m=0}^{k-1}e_p\Big(\sum_{j=0}^{s-1}h_j v_{m+j}(x)\Big)-\sum_{m=N-k}^{N-1}e_p\Big(\sum_{j=0}^{s-1} h_j v_{m+j+k}(x)\Big)\bigg| \\
&\leq 2k.
\end{align*}
%\[\bigg|\sum_{m=0}^{N-1}e_p\Big(\sum_{j=0}^{s-1}h_j v_{m+j}(x)\Big)-\sum_{m=0}^{N-1}e_p\Big(\sum_{j=0}^{s-1} h_j v_{m+j+k}(x)\Big)\bigg| \le 2k\]
Therefore, for any integer $K \geq 1$ it holds that
\begin{equation*}
\begin{split}
K\bigg|\sum_{m=0}^{N-1}e_p\Big(\sum_{j=0}^{s-1}h_j v_{m+j}(x)\Big) \bigg| &\le \bigg|\sum_{k=0}^{K-1} \sum_{m=0}^{N-1} e_p\Big(\sum_{j=0}^{s-1}h_j v_{m+j+k}(x)\Big) \bigg|+\sum_{k=0}^K 2k\\
&\le \sum_{m=0}^{N-1} \bigg|\sum_{k=0}^{K-1} e_p\Big(\sum_{j=0}^{s-1}h_j v_{m+j+k}(x)\Big)\bigg| + O(K^2).
\end{split}
\end{equation*}
Combining this with \eqref{eq:bound1}, and noting that $\sum_{0<\| \vh \|_\infty \le H} \frac{1}{\rho(\vh)}\ll (\log H)^s$, then gives
\begin{equation}
D_{s,\psi}(N;x) \ll \frac{K}{N}(\log N)^s + \frac{1}{N}R(N,K,x)
\label{eq:bound4}
\end{equation}
where
\begin{equation}
R(N,K,x) = \frac{1}{K}\sum_{0<\| \vh \|_\infty \le N/2} \frac{1}{\rho(\vh)}\sum_{m=0}^{N-1}\bigg|\sum_{k=0}^{K-1}e_p\Big(\sum_{j=0}^{s-1}h_jv_{m+j+k}(x)\Big)\bigg|.
\label{eq:bound2}
\end{equation}

We now average over initial values $x$. By Cauchy--Schwarz one has
\[\left( \sum_{x\in\F_p^n}\bigg|\sum_{k=0}^{K-1}e_p\Big(\sum_{j=0}^{s-1}h_jv_{m+j+k}(x)\Big)\bigg| \right)^2 \le p^n \sum_{x\in\F_p^n} \bigg|\sum_{k=0}^{K-1} e_p\Big(\sum_{j=0}^{s-1}h_j v_{m+j+k}(x)\Big)\bigg|^2.
\]
Inserting this and the bound from Lemma \ref{lem:2nd-moment} into \eqref{eq:bound2} gives
\begin{equation}
\sum_{x \in \F_p^n} R(N,K,x) \ll_{s,n} Np^n (K^{-1/2}+p^{-1/4})(\log N)^s.
\label{eq:bound3}
\end{equation}

Now, let $N_j = 2^j$ and $K_j = \ceil{\Delta^{-2/3}N_j^{2/3}}$ for $j \in \set{ 0,1,\dots,\ceil{\log_2 p^n}}$. Let $\Omega_j\subseteq \F_p^n$ be the set of $x$ for which
\[ R(N_j,K_j,x) \ge C_{s,n}\Delta^{-1}N_j(K_j^{-1/2}+p^{-1/4}) (\log N_j)^s \log p^n,\]
where $C_{s,n}$ is the implied constant in \eqref{eq:bound3}. By \eqref{eq:bound3} we must have that $|\Omega_j| \le \Delta p^n/\log p^n$. Setting $\Omega = \cup_j \Omega_j$ we then have $|\Omega| \le \Delta p^n$, and for $x \not\in \Omega$ it holds that
\begin{equation}
R(N_j,K_j,x) \le C_{s,n} \Delta^{-1}N_j(K_j^{-1/2}+p^{-1/4}) (\log N_j)^s \log p^n
\label{eq:bound5}
\end{equation}
for all $j \le \ceil{\log_2 p^n}$.

Given $N$ such that $1\le N \le p^n$, take $\nu\in\N$ such that $N_{\nu-1}\le N < N_{\nu}$. By \eqref{eq:bound4} we have
\[ D_{s,\psi}(N;x) \ll \frac{K_\nu}{N_\nu}(\log N_\nu)^s + \frac{1}{N_\nu}R(N_\nu,K_\nu,x),\]
and so for $x \not\in \Omega$ it holds that 
\[ D_{s,\psi}(N;x) \ll (\Delta^{-2/3}N^{-1/3} + p^{-1/4}\Delta^{-1})(\log N)^s \log p^n\]
by \eqref{eq:bound5}, completing the first bound in the theorem.

For $D_{s,\psi}^n(N;x)$ we also apply Theorem \ref{thm:koksma} with $H = \ceil{N/2}$ to get
\begin{equation*}
D_{s,\psi}^n(N;x) \ll \frac{1}{N}+\frac{1}{N}\sum_{0< \| \vh \|_\infty \le N/2} \frac{1}{\rho(\vh)}\Big| \sum_{m=0}^{N-1} e_p\Big(\sum_{j=0}^{s-1}\vh_j\cdot \vu_{m+j}(x)\Big)\Big|,
\end{equation*}
where now $\vh = (\vh_0,\dots,\vh_{s-1})$ and $\vh_j = (h_{j,1},\dots,h_{j,n})$ for $j \in \set{0,\dots,s-1}$. Observe that
\[ \sum_{j=0}^{s-1} \vh_j\cdot \vu_{m+j}(x) = \sum_{j=0}^{s-1}\sum_{i=1}^{n} h_{j,i}v_{(m+j)n+i}(x) = \sum_{k=0}^{sn-1}h_j' v_{mn+k}(x), \]
where $h_k' = h_{j,i}$ if $k = nj+i$, $0\le i \le s-1$. We may therefore bound all sums exactly as before, with the only difference being that $sn$~replaces $s$.
\end{proof}

%%%%%%%%%%%%%%%%%%%%%%%%%%%%%%%%%%%%%%%%%%%%%%%%%%%
%%%%%%%%%%%%%%%%%%%%%%%%%%%%%%%%%%%%%%%%%%%%%%%%%%%

\section{The computational complexity of fractional jump sequences} \label{computation}

Let $\Psi$ be a transitive automorphism of $\PP^n$, and let $\psi$ be its fractional jump. We now want to establish the computational complexity of computing the $m$-th term of the sequence $\{ \psi^m (0) \}_{m \in \N}$. In particular in this section we will show that computing a term in our sequence is less expensive than computing a term of a classical inversive sequence of the same bit size.

Fix notations as in Section \ref{explicit}. For simplicity, let us restrict to  the case  in which $q$ is prime.
Let us first deal with the regime in which $q$ is large (which is the regime in which we got the discrepancy bounds in Section \ref{discrepancy}), so that most of the computations will be performed for points in $U_1$.
If one chooses $\Psi$ in such a way that the coefficients of the $F_j$'s are small (this is possible for example by taking $\Psi$ as the companion matrix of a projectively primitive polynomial with small coefficients), so that also the coefficients of the $f_j^{(1)}$'s are small, the multiplications for such coefficients cost essentially the same as sums. Therefore the computational cost of computing the $m$-th term of the sequence (given the $(m-1)$-th term) is reduced to the cost of computing $n$ multiplications in $\F_q$ and one inversion in 
$\F_q$ (as all the denominators of the $f_j^{(1)}$'s are equal). Let $M(q)$ (resp. $I(q)$) denote the cost of one multiplication (resp. inversion) in 
$\F_q$. The total cost of bit operations involved to compute a single term in the sequence is then \[C^{\text{new}}(q,n)=n M(q)+ I(q).\]

Using the fast Fourier transform for multiplications \cite{bib:SchStr71} and the
extended Euclidean algorithm for inversion \cite{bib:schonhage71}
one gets
\begin{align*}
M(q) &= O(\log(q)\log\log(q) \log \log \log (q)), \\
I(q) &= O(M(q)\log\log(q)).
\end{align*}

Let us compare this complexity with the complexity of computing the $m$-th term of an inversive sequence of the form $x_{m}=a/x_{m-1}+b$ over $\F_p$. The correct analogue is obtained when $q^n$ has roughly the same bit size as $p$. If one chooses $a,b$ small, one obtains that the complexity of computing $x_m$ is essentially the complexity of computing only one inversion modulo $p$, which is 
\[C^{\text{old}}(p)=O(\log(p)[\log \log(p)]^2 \log \log \log(p)).\] 
Now, since $q,n$ are chosen in such a way that $q^n$ has roughly the same bit size as $p$, we can write $C^{\text{old}}(q,n)=C^{\text{old}}(p)$.
It is easy to see that up to a positive constant we have
\[\frac{C^{\text{new}}(q,n)}{C^{\text{old}}(q,n)}\leq \frac{1}{\log \log q}+\frac{1}{n},\] which goes to zero as $n$ and $q$ grow.

It is also interesting to see that with our construction we have the freedom to choose $q$ relatively small and $n$ large (so that again one gets $q^n\sim p$). 
In this case one can see that ${C^{\text{new}}(q,n)}/{C^{\text{old}}(q,n)}$ goes to zero as $([\log(n)]^2\log \log(n))^{-1}$.
If one tries to do something similar with an ICG (i.e. reducing the characteristic but keeping the size of the field large), one would anyway have to compute an inversion in $\F_{q^n}$ which costs $O(n^2)$ $\F_q$-operations, see \cite[Table 2.8]{bib:MVOV96}, while in our case we would only need to invert one element and multiply $n$ elements in $\F_q$, which costs $(n+1)$ $\F_q$-operations.

%%%%%%%%%%%%%%%%%%%%%%%%%%%%%%%%%%%%%%%%%%%%%%%%%%%
%%%%%%%%%%%%%%%%%%%%%%%%%%%%%%%%%%%%%%%%%%%%%%%%%%%

\section{Conclusion and further research}
Using the theory of projective maps, we provided a general construction for full orbit sequences over $\mathbb A^n$. Our theory generalises the standard construction for the inversive congruential generators.
Let us summarise the properties of fractional jump sequences obtained in this paper:
\begin{itemize}
\item We completely characterise the full orbit condition for such sequences (Theorem \ref{transitivity_characterisation}).
\item In dimension $1$ they cover the theory of ICG sequences.
\item In dimension greater than $1$, they are automatically full orbit whenever $(q^{n+1}-1)/(q-1)$ is a prime number, which is something that can never occur in the case of ICG sequences, as $2$ always divides $q+1$ when $q$ is odd.
\item In any dimension, they enjoy the same discrepancy bound as the one of the ICG, so they appear to be a good source of pseudorandomness both when one desires a one dimensional sequence of pseudorandom elements  or if one desires a stream of $n$-dimensional pseudorandom points (Theorem \ref{thm:discrepancy}). 
\item  They are very inexpensive to compute: for $n>1$ computations are asymptotically quicker than the ones of an ICG sequence, as described in Section \ref{computation}. The moral reason for this is that at each step the ICG generates  a $1$-dimensional pseudorandom point with exactly one inversion in $\F_q$, on the other hand at each step our construction generates an $n$-dimensional point of $\F_q$ (again, using only one inversion).
\end{itemize}

Some research questions arising are the following.
\begin{enumerate}
\item 
As our bound on the discrepancy holds for any transitive non-affine fractional jump sequence, can one build special fractional jump sequences having strictly better discrepancy bounds with respect to the one of the ICG?
\item What happens if one replaces the finite field with a finite ring? Can we extend the fractional jump construction to this case?
\item Can the notion of fractional jump be extended to more general objects such as quasi-projective varieties and produce competitive behaviours as in the projective space setting?
\end{enumerate}

\section*{Acknowledgments}

The authors would like to thank Violetta Weger for checking the preliminary version of this manuscript.
The third author would like to thank the Swiss National Science Foundation grant number 171248.

\bibliographystyle{abbrv}
\bibliography{Bibliography.bib}

\end{document}